\documentclass[a4paper,10pt]{amsart}
\usepackage{amsmath}\usepackage{amsthm}
\usepackage{latexsym}
\usepackage[psamsfonts]{amssymb}
\usepackage{amscd,mathrsfs}
\usepackage[colorlinks=true,linkcolor=black,citecolor=black]{hyperref}
\newtheorem*{cor*}{Corollary}
\newtheorem{theorem}{Theorem}
\newtheorem{lemma}[theorem]{Lemma}
\newtheorem*{lemma*}{Lemma}
\newtheorem*{conj*}{Conjecture}
\newtheorem*{prop*}{Proposition}
\newtheorem{prop}[theorem]{Proposition}
\newtheorem{cor}[theorem]{Corollary}

\theoremstyle{definition}
\newtheorem{remark}[theorem]{Remark}
\newtheorem*{remark*}{Remark}
\newtheorem*{question*}{Question}
\newtheorem{example}[theorem]{Example}
\renewcommand{\H}{\mathcal{H}}

\newcommand{\comment}[1]{} 
\newcommand{\wT}{\widetilde{T}}
\def\im{\operatorname{Im}} \def\ker{\operatorname{Ker}} \def\stab{\operatorname{Stab }}
\def\sgn{\operatorname{sgn}}  

 \def\Q{\mathbb{Q}}\def\R{\mathbb{R}}\def\Z{\mathbb{Z}}
\def\le{\leqslant} \def\ge{\geqslant}
\def\tr{\operatorname{tr}}
\def\SL{\mathrm{SL}} \def\PSL{\mathrm{PSL}}\def\GL{\mathrm{GL}}

\def\X{\mathbf{X}}
  
\def\e{\varepsilon_\Gamma}\def\w{w_\Gamma}
\def\DD{\Delta} \def\G{\Gamma}
\def\SS{\Sigma}
\def\dd{\delta} \def\ss{\sigma}
  
\def\g{\gamma}

\def\ppf#1#2{\frac{\partial #1}{\partial #2}}
\def\ppsf#1#2{\frac{\partial_\SS #1}{\partial #2}}
\def\pp#1#2{\partial #1\slash \partial #2}
\def\pps#1#2{\partial_\SS #1\slash \partial #2}
\def\be{\begin{equation}}  \def\ee{\end{equation}}
\def\ov#1{\overline{#1}}

\def\vp{\varphi}\def\wG{\widetilde{\Gamma}}\def\wL{\widetilde{L}}
\def\rar{\rightarrow}\def\lrar{\longrightarrow}
\def\res{\operatorname{res}}\def\inf{\operatorname{inf}}
\def\hom{\operatorname{Hom}}\def\tg{\mathrm{tg}}
\def\vv{\mathbf{v}}\def\deg{\operatorname{deg}}\def\ind{\operatorname{Ind}}
\def\la{\langle}\def\ra{\rangle}

\title{A trace formula for Hecke operators on Fuchsian groups}
\author{Alexandru A. Popa}

\address{Institute of Mathematics of the Romanian Academy\\
P.O. Box 1-764, Bucharest RO-70700, Romania}
\email{aapopa@gmail.com}
\thanks{The author was partially supported by the
CNCS-UEFISCDI grant PN-III-P4-ID-PCE-2020-2498.}
\begin{document}

\begin{abstract}
In this paper we give a trace formula for Hecke operators acting on the
cohomology of a Fuchsian group of finite covolume $\Gamma$ with coefficients in a module $V$.
The proof is based on constructing an operator whose trace on $V$ equals the
Lefschetz number of the Hecke correspondence on cohomology, generalizing the operator introduced
together with Don Zagier for the modular group.
\end{abstract}

\maketitle

\section{Introduction}

Let $\G$ be a discrete, finite covolume subgroup of $\PSL_2(\R)$,
and $\SS$ a double coset of $\G$ contained in the commensurator
$\widetilde{\G}\subset\GL_2^+(\R)$. For $V$ a finite dimensional $\widetilde{\G}$-module,
we consider the action of the Hecke operator defined by the double coset $\SS$,
which we denote by $[\SS]$, on the cohomology groups
$H^n(\G,V)$,  $n=0,1,2$. The goal of this paper is to prove a trace formula of the type
\be\label{e_TF0}
\sum_i (-1)^i \tr\big( [\SS], H^i(\G,V) \big)=\sum_{X\subset \SS} \e(X)\cdot \tr(M_X, V) ,
\ee
where the sum is over $\G$-conjugacy classes $X$ with representatives $M_X\in X$,
and  $\e(X)$ are certain conjugacy class invariants. We assume throughout that $V$ is a finite 
dimensional vector space, over a field in which the orders of finite elements in $\G$ are invertible. 

When $V=V_{k-2}$, the $\GL_2(\R)$-module of homogeneous polynomials of degree $k-2$,
we obtained such formulas with Zagier in~\cite{PZ} for the modular group,
and in~\cite{P} for congruence subgroups,
under a mild hypothesis on $\SS$. By
a version of the Eichler-Shimura correspondence  we have a Hecke-equivariant isomorphism
\[H^1(\G, V_{k-2})\simeq M_{k}(\G)\oplus S_k(\G)^c
\]
where $M_{k}(\G)$ is the space of modular forms and $S_k(\G)^c$ the space of anti-holomorphic cusp
forms. The trace on $H^0$ is easy to compute  (and $H^2$ is trivial for congruence subgroups), 
and one gets explicit formulas on the space of modular forms. 

Let $\Q[\SS]$ be the vector space of finite linear combinations of elements of $\SS$, viewed as
a left and right $\Z[\G]$-module. For the modular group, we showed in~\cite{PZ} that
the right-hand-side of~\eqref{e_TF0} arises as the trace on $V$ of a specific element $\wT_\SS\in\Q[\SS]$,
giving the action of the Hecke operator $[\SS]$ on the space of period polynomials associated
to modular forms. The same operator introduced for the modular group was used in~\cite{P} for congruence subgroups, because formula~\eqref{e_TF0} behaves well with respect to induction: if $\G'\subset \G$
is a finite index subgroup and $V$ a $\G'$-module, then a formula of type~\eqref{e_TF0} for $\G$ 
and the induced module $\ind_{\G'}^\G V$ implies a formula for $\G'$ and $V$, 
by the Shapiro lemma and an argument as in~\cite[Sec. 4]{P}. When $V$ is also a $\G$-module, the
trace of elements $M\in \SS$ on $\ind_{\G'}^\G V$ is easy to compute, and in~\cite{P2}
we obtain from~\eqref{e_TF0} uniform formulas for the trace of a composition of Atkin-Lehner and Hecke operators on  $S_k(\G_0(N))$, which hold without any restriction on the index of the operators.

The present paper is motivated by generalizing a key ingredient of~\cite{PZ}
to Fuchsian groups $\G$ of finite covolume:
find elements $\wT_\SS\in \Q[\SS]$ such that the left hand side of~\eqref{e_TF0} is given by $\tr(\wT_\SS,V)$.
Such a group $\G$ admits non-trivial Hecke operators if and only if it is of arithmetic type,
namely $\G$ is commensurable
either with a congruence subgroup of the modular group (if $\G$ has cusps), or with a congruence subgroup of
the unit group in a definite quaternion algebra (if $\G$ has no cusps).
In this paper we treat both cases in a uniform manner, and construct explicitly such elements~$\wT_\SS$,
under an assumption on $V$ when $\G$ has no cusps.
 Even for the modular group, our method is different from that used
in~\cite{PZ}, where we used the theory of period polynomials. Rather, we use the
action of Hecke operators on the algebraically defined cohomology groups $H^i(\G,V)$ using a presentation of $\G$
in terms of generators and relations.

As known to Poincar\'e,  a finite covolume Fuchsian group $\G$ has a presentation in terms of hyperbolic elements
$\g_1, \g_2, \ldots, \g_{2g-1}, \g_{2g}$, elliptic elements $s_1, \ldots, s_l$ generating the stabilizers of all
non-equivalent elliptic points, and  parabolic elements $t_1, \ldots, t_h$ belonging to non-conjugate maximal parabolic subgroups, subject to the relations
  \be \label{e_fuchsian}  [\g_1, \g_2] \cdot \ldots \cdot[\g_{2g-1}, \g_{2g}]\;
 s_1\cdot\ldots \cdot s_l\cdot t_1\cdot\ldots \cdot t_h=1,
  \quad s_i^{m_i}=1,  \ m_i\ge 2  \;. \ee
The signature $(g;m_1,\ldots, m_l; h)$ is an invariant of the group, and by the Gauss-Bonnet formula
  \be\label{e_chi}
  2g-2+\sum_{j=1}^l \Big(1-\frac{1}{m_j}\Big)+h=\frac{|\G\backslash\H|}{2\pi},
  \ee
where~$|\G\backslash\H|$ is the area of a fundamental domain for $\G$ acting on the upper half plane $\H$, 
with respect to the standard hyperbolic metric.
Moreover, for every signature such that the left side is positive, there exists such a group $\G$.

To state the main result, in this introduction we assume $\G$ has cusps ($h\ge 1$). Replace $t_h$ in
terms of the other generators eliminates one relation, and consequently we view $\G$ as a group with
generators $g_1, \ldots, g_n$ with
$n=2g+\ell+h-1$, and relations $g_i^{m_i}=1$ for $i=1,\ldots, \ell$. For every $g\in \G$, in the group algebra
$\Z[\G]$ we have a decomposition
\[1-g=\sum_{i=1}^n (1-g_j) \ppf{g}{g_j}, 
\]
for some elements $\pp{g}{g_i}\in \Z[\G]$. These elements are not unique, unless $\G$ is free ($\ell=0$), 
and in the case of a free group they are called Fox derivatives. They were introduced by Ralph Fox in the 1950s~\cite{F53}, who developed a free differential calculus with applications to topology and knot theory. 

We will also need a generalization of these elements $\pps{g_i}{g_j}\in \Z[\SS]$, defined by 
\[
T_\SS(1-g_i)=\sum_{j=1}^n (1-g_j) \ppsf{g_i}{g_j}
\]
where $T_\SS=\sum_{\tau\in\G \backslash \SS} \ov{\tau}$ is the usual Hecke operator acting on modular forms, 
with $\ov{\tau}$ fixed coset representatives. Note that for $\SS=\G$ the elements $\pps{g_i}{g_j}$
reduce to the usual Fox derivatives. Define 
\[
\wT_\SS=T_\SS-\sum_{i=1}^n \ppsf{g_i}{g_i}+\sum_{i=1}^\ell \pi_i \cdot \ppsf{g_i}{g_i} \in \Q[\SS],
\]
where $\pi_i=(1+g_i+\ldots+g_i^{m_i-1})/m_i$ are idempotents associated to the elliptic elements $g_i$,
$i=1,\ldots, \ell$. In Proposition~\ref{prop_cusps} we prove the following result. 
\begin{prop*}
For $\G$ a finite covolume Fuchsian group with cusps and $V$ a finite 
dimensional $\wG$-module over a field in which the orders of elliptic elements are invertible we have:
\[
\sum_{i=0,1} (-1)^i \tr([\SS], H^i(\G,V))=\tr(\wT_\SS, V).
\]
\end{prop*}
Writing $\wT_\SS=\sum_{M} c_M M$, with finitely many non-zero coefficients $c_M$, formula~\eqref{e_TF0}
immediately follows by taking
\be\label{e_eps}
\e(X)=\sum_{M\in X} c_M
\ee
for $X$ a conjugacy class. In Theorem~\ref{thm_main} we also 
show that $\e(X)$ do not depend on the choices made in defining $\wT_\SS$.

We construct such elements also when $\G$ has no cusps (Proposition~\ref{prop_compact}),
assuming it has no elliptic elements either and under some restriction on the module $V$,
which is likely not necessary (see Remark~\ref{rem_compact}). In constructing the elements~$\wT_\SS$
and proving their properties, we use results of Lyndon~\cite{L50}, so our approach can be seen as belonging to combinatorial group theory.

We now discuss the conjugacy class invariants $\e(X)$.
For $M\in \SS$, let $n(M)$ be the number of fixed points of $M$ in $\H\cup\{\text{cusps of $\G$}\}$, 
and let $\G_M$ denote the centralizer of $M$ in $\G$. We define:
\[
\w(M)=\frac{ (-1)^{n(M)+1}}{|\G_M|}  \text{ if $M$ not scalar},\qquad
\w(M)=-\frac{|\G\backslash\H | }{2 \pi} \text{ if $M$ scalar},
\]
with the understanding that $\w(M)=0$ if $\G_M\ne\G$ is infinite.
More explicitly, it can be shown (e.g.~\cite[Proof of Theorem 2]{O77}) that any
$M\in \widetilde{\G}$ for which $\G_M$ is finite is either: elliptic  (that is $\tr^2M<4 \det M$) so $n(M)=1$;
or it is split hyperbolic (that is $\tr^2M>4 \det M$ and $M$ fixes two distinct cusps of $\G$) so $n(M)=2$ and
$|\G_M|=1$.

Since $n(M)$ and $|\G_M|$ are conjugacy class invariants, we define $\w(X)=\w(M)$ for a conjugacy class $X\subset \SS$ with representative $M$.
We have shown in~\cite{PZ,P} that the trace formula~\eqref{e_TF0} holds for
\be\label{e_conj}\e(X)=\w(X),
\ee
if $\G$ is any finite index subgroup of $\G_1=\SL_2(\Z)$, and $\SS$ is a double coset satisfying
$|\G\backslash \SS|=|\G_1\backslash \G_1\SS|$.

We expect formula~\eqref{e_conj} to hold for all finite covolume Fuchsian groups.
It is an interesting open problem to prove it using only the elements $\wT_\SS$ defined
in this paper, which do not depend on the module $V$. We can only do it in the case
of $\SL_2(\Z)$ in Section~\ref{sec_modular}, by relating the element introduced here to the explicit one constructed in~\cite{PZ}, for which we explicitly computed the coefficient sums~\eqref{e_eps}.

One case when we can prove formula~\eqref{e_conj}  is for the trivial Hecke operator $\SS=\G$.
Under some restrictions stated precisely in Corollary~\ref{cor_dim}, we obtain
a formula for the Euler-Poincar\'e characteristic of the $\G$-module~$V$:
\[\sum_i (-1)^i \dim H^i(\G,V)=-\frac{|\G\backslash \H|}{2\pi}\dim V + 
\sum_{M\in E(\G)} \frac{1}{|\G_M|}\tr(g, V) ,
\]
where $E(\G)$ is a set of representatives for elliptic conjugacy classes in $\G$.
However in this case the formula is known for a much larger class of groups,
by work of Bass~\cite{Ba} and Brown~\cite{Br}. The constants $-|\G\backslash \H|\slash 2\pi$
and $1\slash|\G_M|$ appearing above
occur there as homological Euler characteristics of the groups $\G$ and $\G_M$, 
as defined for example by Serre~\cite{Se}. It would be interesting to extend
the techniques of the papers above to prove a formula for the Lefschetz number of a Hecke correspondence 
of the type~\eqref{e_TF0} for a more general class of arithmetic groups. In particular,
for Fuchsian groups of finite covolume, it would be interesting to give a more conceptual proof
of~\eqref{e_conj}, other than by constructing an explicit element $\wT_\SS$ as in~\cite{PZ}.

Formulas for Lefschetz numbers of Hecke correspondences
are known by the topological trace formula of Goresky and MacPherson, 
developed over the course of a decade~\cite{GM1,GM2}. 
Such formulas are proved for~$\G$ any arithmetic subgroup of a reductive group 
and the algebraic group cohomology replaced by other geometric cohomology theories.
The formulas involve local contributions from the fixed point varieties of the correspondence,
which lump together contributions from the various conjugacy classes in $\G$. It would be interesting 
to investigate whether a simple minded version stated in terms of group cohomology as in~\eqref{e_TF0}
holds in this more general setting.

\emph{Acknowledgements.} I would like to thank Vicen\c tiu Pa\c sol for helpful discussions
while preparing this paper.



\section{Hecke operators on the first cohomology of groups defined by generators and relations}

\subsection{Groups defined by generators and relations}\label{sec_genrel}

Here we review some results of Lyndon~\cite{L50}, who used them to compute the cohomology of a
group defined by a single relation.

Let $\G=\la g_1, \ldots, g_n : r_1, \ldots , r_m\ra$ be a group given in terms of a presentation with generators
$g_i$, and relations $r_j=1$. We view also $\G= F\slash R$ as a quotient of a free group $F$ with generators
denote by the same symbols $g_1,\ldots, g_n$,
by the normal subgroup $R$ generated by the elements $r_j\in F$. We identify elements of $\G$ with
right cosets $Rg$ of $F\slash R$.
To keep the notation simple, we use the same symbols
to denote the elements of $F$ and their projections onto $\G$, as the group will be clear from the context.

Let $\Z[F]$ be the group algebra of $F$.
For  $g\in F$ there are unique elements $\partial g\slash \partial g_i\in \Z[F]$ such that
\be\label{e_decg}
1-g=\sum_{i=1}^n (1-g_i) \frac{\partial g}{\partial g_i}.
\ee
The map $\partial \slash \partial g_i: F\mapsto \Z[F]$ is a cocycle for the right action
of $F$ on $\Z[F]$  (traditionally called a derivation of $F$), namely it satisfies
\[
\frac{\partial gh}{\partial g_i}=\frac{\partial h}{\partial g_i}+\frac{\partial g}{\partial g_i}\cdot h.
\]
It is the unique derivation of $F$ such that
$$\ppf{g_j}{g_i}=\dd_{i,j}, \quad  \ppf{g_i^{-1}}{g_j}=-\ppf{g_i^{-1}}{g_j} \cdot g_j , $$
where $\dd_{i,j}$ is the Kronecker delta function.

The group algebra $\Z[\G]$ is identified with the quotient of $\Z[F]$ by the right ideal generated
by elements $(r-1) a$ with $r\in R, a\in Z[F]$. For $g\in F$ we can define uniquely the
partial derivatives $\pp{g}{g_i}\in \Z[\G]$ by means of this projection. For $g\in \G$,
we define
$\pp{g}{g_i}\in \Z[\G]$ to be any elements such that~\eqref{e_decg} holds.
However these elements are no longer unique in $\Z[\G]$: for each $r\in R$ we have
\[0=1-r=\sum_{i=1}^n (1-g_i) \frac{\partial r}{\partial g_i}
\]
(here the symbol $r$ is viewed both as an element of $F$, and as its projection on $\G$ which equals the
identity element 1).
  The failure of uniqueness is precisely described in terms of the relators~$r_i$ by Lyndon~\cite[Lemma 5.1]{L50},
which we state below.
For $g\in F$, it is convenient to denote by $\nabla {g}=(\pp{g}{g_1},\ldots, \pp{g}{g_n})\in\Z[\G]^n$.

\begin{prop}[\cite{L50}] \label{prop_lyndon}
Let $\X=(X_1, \ldots , X_n)\in \Z[\G]^n$. In $\Z[\G]$ we have
 \[
 \sum_{i=1}^n (1-g_i) \cdot X_i=0 \text{ if and only if }
 \X=\sum_{j=1}^m  \nabla {r_j} \cdot k_j  \text{ for some $k_j\in \Z[\G] $}.
 \]
 \end{prop}
\begin{example} For the modular group $\G=\PSL_2(\Z)$, the proposition reduces to the aciclicity lemma~\cite[Lemma 2]{CZ}. The group $\G$ is the quotient of the free group on two elements $S,U$,
by the relations $r_1=S^2, r_2=U^3$. We have $\partial{r_1}\slash \partial S=1+S$,
$\partial{r_2}\slash \partial U=1+U+U^2$, and $\partial{r_1}\slash \partial U=0$,
$\partial{r_2}\slash \partial S=0$. Therefore in this case the proposition shows that
\[(1-S)Y=(1-U)Z \Leftrightarrow Y\in (1+S) \Z[\G],\ Z\in (1+U+U^2) \Z[\G],
\]
that is $\im(1-S)\cap\im(1-U)=\{0\}$. Since  $\ker(1+S)=\im(1-S)$, $\ker(1+U+U^2)=\im(1-U)$, we recover
the aciclicity lemma.
\end{example}

\comment{
Assume that $X_i\in \Z[\G]$ satisfy
$
\sum_{i=1}^n (1-g_i) \cdot X_i=0.$

 Then there are $k_j\in\Z[\G]$ such that
 $$X_i=\sum_{j=1}^m  \frac{\partial{r_j}}{\partial{g_i}} k_j.$$}

Let $V$ be a right $\G$-module, which we view also as a right $F$-module
with $R$ acting trivially on $V$. If $\varphi:\G\rightarrow V$ is a cocycle,
namely $\vp(gh)=\vp(h)+\vp(g)|h$, we view it as a cocycle on $F$ trivial on $R$
(by the inflation map $H^1(\G,V)\rar H^1(F,V)$). Note that by the cocycle relation
we have $\vp(rg)=\vp(gr)=\vp(g)$ for $r\in R, g\in F$.

The following lemma can also be found in~\cite[Sec. 3, eq. (5)]{L50}.

\begin{lemma} \label{lem_cocycle}
Let $V$ be a (right) $\G$-module, and $\varphi:\G\rightarrow V$ be a cocycle. Viewing $\vp$ as
a cocycle on $F$ as above, we have for $g\in F$:
$$\varphi(g)=\sum_{i=1}^n\varphi(g_i)\left|\frac{\partial{g}}{ \partial{g_i}}\right.\;.$$
 \end{lemma}
\begin{proof}
The statement follows from~\eqref{e_decg} by induction on the length of $g$ as a product in the generators $g_i$.
Indeed we have $1-g_ig=1-g +(1-g_i)g$, and $\varphi(g_i g)=\varphi(g_i)|g+\varphi(g)$.
\end{proof}

\subsection{Double coset operators on cohomology}

Let $\G$ be a group and $\wG$ its commensurator inside an ambient group.
Let $\SS\subset \wG$ be a double coset of $\G$, so that the number of
cosets $\G\backslash \SS$ is finite. Let $V$ be right~$\widetilde{\G}$-module.

In this paper we view the cohomology groups defined algebraically as
$H^i(\G,V)=Z^i(\G,V)\slash B^i(\G,V)$,
where $Z^i(\G,V)$ are (inhomogeneous) cocycles and $B^i(\G,V)$ coboundaries.
The double coset operator $[\SS]$ acts on cohomology as follows.
We fix representatives $\ov{\tau}\in \SS$ for cosets $\tau\in \G\backslash \SS$, and
for $g\in\G$, we let $\g_{\tau,g}\in\G$ be the unique element such that
\be\label{e_taug}
\ov{\tau}g=\g_{\tau,g} \ov{\tau g}.
\ee
If $\vp:\G^i\rar V$ is a a cocycle, following \cite{Sh} we define
\[
\vp|[\SS](h_1, \ldots , h_i)= \sum_{\tau\in  \G\backslash \SS}
\vp(\g_{\tau h_1^{-1},h_1}, \g_{\tau h_1^{-1}h_2^{-1},h_2},  \ldots,
\g_{\tau h_1^{-1}\cdot \ldots \cdot h_i^{-1},h_i} )|\ov{\tau}.
\]
Then $\vp|[\SS]$ is a cocycle, whose cohomology class is
independent of the choice of representatives~$\ov{\tau}$.

We are mostly interested in $H^0$ and $H^1$. We have $H^0(\G,V)=V^\G$, the space of
invariants of $\G$, and for $v\in V^\G$ we have
\[v|[\SS]=\sum_{\tau \in\G\backslash \SS}v|\ov{\tau}.
\]
If $\vp:\G\rightarrow V$  is a cocycle,  we have
  \be\label{e_coset}\vp|[\SS](g)= \sum_{\tau \in\G\backslash \SS} \vp(\g_{\tau g^{-1},g})|\ov{\tau} \;.\ee

\subsection{Double coset operators for groups generated by generators and relations}
Assume now $\G=F/R$ is given in terms of generators and relations as in section~\ref{sec_genrel}.
Let $\Z[\SS]$ be the set of finite linear combinations
$\sum n_j \ss_j$ with $\ss_j\in \SS$, $n_j\in\Z$. We view  $\Z[\SS]$ as a left and right $\Z[\G]$-module, and also
as a left and right $\Z[F]$-module by having $R$ act trivially by left and right multiplication.

Fix representatives $\ov{\tau}\in \SS$ for cosets $\tau\in \G\backslash \SS$, and define
\[T_\SS=\sum_{\tau\in \G\backslash \SS} \ov{\tau}\in \Z[\SS].\]
For a fixed generator $g_i$, the map $\tau\mapsto \tau g_i$ permutes the cosets
$\tau\in \G\backslash \SS$, so we can write:
\[ \begin{aligned}
  T_\SS(1-g_i)&=\sum_{\tau\in \G\backslash\SS}\ov{\tau}-\ov{\tau g_i^{-1}}g_i=
  \sum_{\tau}  (1-\g_{\tau g_i^{-1}, g_i})\ov{\tau}\\
  &= \sum_{j=1}^n (1-g_j)\sum_\tau\ppf{\g_{\tau g_i^{-1}, g_i}}{g_j} \ov{\tau},
  \end{aligned}
  \]
 where $\g_{\tau,g}$ is defined as in~\eqref{e_taug}.
  We conclude that
\be \label{e_hecke}T_\SS (1-g_i)=\sum_{j=1}^n (1-g_j) \ppsf{g_i}{g_j} \;, \text{ where }
 \ppsf{g_i}{g_j}:=\sum_{\tau \in \G\backslash \SS } \ppf{\g_{\tau g_i^{-1}, g_i}}{g_j} \cdot \ov{\tau} \in\Z[\SS].
 \ee
The elements  $\pps{g_i}{g_j}$ depend on the choice of representatives $\ov{\tau}$ for the cosets $\tau$,
and on the non-unique choices of $\pp{\g_{\tau g_i^{-1}, g_i}}{g_j}$,
but we omit the dependence from the notation.
For arbitrary $g\in F$ we define
\[
\ppsf{g}{g_i}=\sum_{j=1}^n \ppsf{g_j}{g_i} \ppf{g}{g_j},
\]
so that an easy computation gives
\[T_\SS (1-g)=\sum_{j=1}^n T_\SS (1-g_j)\ppf{g}{g_j}=\sum_{i=1}^n (1-g_i) \ppsf{g}{g_i}.
\]
Note that for $\SS=\G$, the trivial double coset,
we have $\pps{g}{g_j}=\pp{g}{g_j}$, so this is an extension to double cosets of the usual
Fox derivatives.
We also denote by $\nabla_\SS g$ the column vector of ``partial $\SS$-derivatives:''
\[
\nabla_\SS g =(\pps{g}{g_1}, \ldots,\pps{g}{g_n} ) \in \Z[\SS]^n,
\]
and by $M_\SS$ the $n\times n$ matrix with entries $\pps{g_j}{g_i}$. In matrix notation, the formula above becomes:
\[
\nabla_\SS g=M_\SS \nabla{g}.
\]
Lemma~\ref{lem_cocycle} generalizes to express the action of double coset operators on cocycles
in terms of partial $\SS$-derivatives.

\begin{lemma}  \label{lem_coc_hecke} If $\vp: \G\rightarrow V$ is a cocycle, for $1\le i \le n$ we have
  \be   \vp|[\SS](g_i)=\sum_{j=1}^n \vp(g_j)\left|\ppsf{g_i}{g_j}\right. . \ee
\end{lemma}
\begin{proof}
This follows immediately from~\eqref{e_coset}, \eqref{e_hecke} and Lemma~\ref{lem_cocycle}.
\end{proof}
We have the following immediate generalization of Proposition~\ref{prop_lyndon} to double cosets.
\begin{prop}\label{prop_lyndonSS}
Let $\X=(X_1, \ldots , X_n)\in \Z[\SS]^n$. We have
 \[
 \sum_{i=1}^n (1-g_i) \cdot X_i=0 \text{ if and only if }
 \X=\sum_{j=1}^m  \nabla {r_j} \cdot k_j  \text{ for some $k_j\in \Z[\SS] $}.
 \]
\end{prop}
\begin{proof}
Write each $X_i=\sum_{\tau\in  \G\backslash\SS} X_{\tau,i}\ov{\tau}$ with
$X_{\tau,i}\in \Z[\G]$. It follows that
$$
\sum_i (1-g_i) X_{\tau,i}=0,
$$
for all $\tau$, and the conclusion follows from Proposition~\ref{prop_lyndon}.
\end{proof}
We apply the proposition to the following situation. Since $r_i=1$ in $\G$ we have
\[
0=T_\SS(1-r_i)=\sum_{j=1}^n (1-g_j)\ppsf{r_i}{g_j},
\]
and we obtain:
\begin{cor}\label{cor_nablaSS}
We have the following decompositions inside~$\Z[\SS]^n$:
\[\nabla_{\SS} r_i=M_\SS \nabla r_i =\sum_{k=1}^m \nabla r_k \cdot T_\SS^{i,k},
\]
for some elements $T_\SS^{i,k}\in \Z[\SS]$, $1\le i,k\le m$.
\end{cor}

\subsection{Double coset operators and the inflation-restriction exact sequence}
\label{sec_infres}
Let $\G=F/R$ as before and $V$ a right $\G$-module. We view $V$ as an $F$-module as well,
by letting $R$ act trivially. Since free groups have cohomological dimension 1,
we have $H^2(F,V)=0$ and the inflation-restriction exact sequence gives:
\be\label{e_infres}
0\lrar H^1(\G, V)\stackrel{\inf}{\longrightarrow} H^1(F,V)
\stackrel{\res}{\longrightarrow}\hom(R,V)^{F}
\stackrel{\tg}{\longrightarrow } H^2(\G,V)
\lrar 0.
\ee
Since $R$ acts trivially on $V$, we have $H^1(R,V)^F=\hom(R,V)^{F}$ consists
of the homomorphisms $f:R\rar V$ such that $f(g^{-1}r g)=f(r)|g$ for $r\in R, g\in F$.

For $\SS$ a double coset of $\G$, the operator $[\SS]$ acts on two of the four cohomology groups in the exact
sequence. We want to define an action of $[\SS]$ on the remaining two groups such that
it is equivariant with respect to all the maps. In this section we concentrate on the first two maps.

Since $F$ is free, we have a map $V^n\rar H^1(F,V)$ taking
a vector $\vv=(v_1\ldots, v_n)$ to the class of the cocycle
$\vp_\vv$ defined by $\vp_\vv(g_i)=v_i$~\cite[Lemma 1.1]{S69}.
We have an exact sequence
\be\label{e_exseqfree}
0\lrar V^\G\lrar V \lrar  V^n \xrightarrow{\;\vv\mapsto [\vp_\vv]\;} H^1(F,V)\lrar  0
\ee
where the second map takes $v\mapsto\left(v|(1-g_1), \ldots, v|(1-g_n)\right)\in V^n$.
We define an action of $[\SS]$ on $V^n$ and on cocycles $\vp:F\rar V$
as follows
\be\label{e_vec_hecke}
(\vv|[\SS])_i:=\sum_{j=1}^n v_j\left| \ppsf{g_i}{g_j}  \right. , \qquad
\vp|[\SS] (g_i):=\sum_{j=1}^n \vp(g_j)\left| \ppsf{g_i}{g_j}\right. .
\ee
By Lemma~\ref{lem_coc_hecke}, the resulting map on $H^1(F,V)$ coincides on the image of $\inf$
with the action of $[\SS]$ on $H^1(\G,V)$,
so the inflation map is equivariant with respect to the two actions.

\begin{prop}\label{prop_free} With the action of $[\SS]$ defined in~\eqref{e_vec_hecke}, we have
\[\tr([\SS], H^1(F,V))=\sum_{i=1}^n \tr(\pps{g_i}{g_i}, V)
-\tr(T_\SS, V)+\tr([\SS],H^0(\G,V))\;.
\]
\end{prop}
\begin{proof}
The kernel of the map $v\mapsto [\vp_\vv]$ is preserved by the map $\vv\mapsto \vv|[\SS]$ defined above,
and we have
\be \label{e_bdact} [v|(1-g_i)]_i |[\SS]= [v|T_\SS|(1-g_i)]_i \; .\ee
Therefore the map \eqref{e_bdact} corresponds to the map $v\mapsto
v|T_\SS$ on $V$ and $V^\G$. Since
\[
\tr(T_\SS, V^\G)=\tr([\SS],H^0(\G,V)),\qquad   \tr([\SS], V^n)=\sum_{i=1}^n \tr(\pps{g_i}{g_i}, V),
\]
the proof is finished by the exact sequence~\eqref{e_exseqfree}.
\end{proof}

The action of $[\SS]$ on the restriction of $F$-cocycles to $R$ is given in the next lemma.
\begin{prop}\label{prop_res_hecke}
If $\vp:F\rar V$ is a cocycle, for $1\le i \le m$ we have
\[\vp|[\SS] (r_i)=\sum_{k=1}^m \vp(r_k)|T_\SS^{i,k},
\]
 with $T_\SS^{i,k}\in \Z[\SS]$ defined in Corollary~\ref{cor_nablaSS}.
\end{prop}
\begin{proof}
 From Lemma~\ref{lem_coc_hecke} and Corollary~\ref{cor_nablaSS} we have:
\[\begin{aligned}
\vp|[\SS] (r_i)&=\sum_{j=1}^n \vp(g_j)\left|\ppsf{r_i}{g_j}\right. =
\sum_{j=1}^n \sum_{k=1}^m\vp(g_j)\left| \ppf{r_k}{g_j} T_\SS^{i,k}\right.\\
&=\sum_{k=1}^m\vp(r_k)| T_\SS^{i,k}.
\end{aligned}\qedhere
\]
\end{proof}

\section{The trace formula on Fuchsian groups}

We now specialize $\G$ to be a discrete, finite covolume subgroup $\G$ of $\PSL_2(\R)$,  so $\G$
has presentation~\eqref{e_fuchsian}.

\subsection{The case $\G$ has cusps}
We assume $h>0$ in the presentation~\eqref{e_fuchsian}. Then we can solve for $t_h$ in terms of the
other generators, which we relabel $g_1,\ldots , g_n$ where $n=2g+\ell+h-1$, so that
the elliptic generators are $g_i=s_i$ for $1\le i\le \ell$.
Therefore $\G=F\slash R$ where $F$ is free on $g_1,\ldots, g_n$ and $R$ is the normal closure
of the relators $r_i=g_i^{m_i}$, for $1\le i\le \ell$.

We assume $V$ is a (right) $k\G$-module where $k$ is any field in which
the orders~$m_i$ of elliptic elements are invertible. It follows that $H^2(\G,V)=0$,
as $\G$ is a free product of the cyclic groups generated by $g_i$, and
one can use the Mayer-Vietoris sequence for free products of groups with amalgamation
due to Swan~\cite[Sec. 2]{S69}.

For $i=1,\ldots, \ell$,
let
\be\label{e_pi}
\pi_i=\frac{\pp{r_i}{g_i}}{m_i}=\frac{1+g_i+\ldots+g_i^{m_i-1}}{m_i} \in \Q[\G]
\ee
be the idempotent associated with $g_i$. 
\begin{prop}\label{prop_cusps}
For $\G$ a finite covolume Fuchsian group with cusps we have:
\[
\sum_{i=0,1} (-1)^i \tr([\SS], H^i(\G,V))=\tr(\wT_\SS, V),
\]
where
\be\label{e_wtss}
\wT_\SS=T_\SS-\sum_{i=1}^n \ppsf{g_i}{g_i}+\sum_{i=1}^\ell \pi_i \cdot \ppsf{g_i}{g_i} \in \Q[\SS].
\ee
\end{prop}
\begin{proof}
If $\ell=0$, the theorem follows immediately from Proposition~\ref{prop_free}, so we assume $\ell>0$.
For $a\in\Q[\G]$, we denote by $\im(a), \ker(a)\subset V$ the image and kernel
of $a$ acting on the right on $V$.

The inflation-restriction sequence~\eqref{e_infres} reduces to three terms, and
we identify
\be\label{e_homisom}
\hom(R,V)^F= \oplus_{i=1}^\ell \im (\pp{r_i}{g_i})
\ee
by the map $f\mapsto [f(r_i)]_i$.
Indeed $r_i=g_i^{m_i}$ and $f(r_i)=f(g_i^{-1} r_i g_i)=f(r_i)|g_i$ so
$f(r_i)\in \ker(1-g_i)=\im (1+g_i+\ldots,g_i^{m_i-1})$
(one inclusion in the equality of subspaces is clear, and if $v=v|g_i$ it follows that
$v=v|\pi_i$ since $m_i$ is invertible in $V$, so $v\in \im(\pi_i)=\im (\pp{r_i}{g_i})$ and the other inclusion follows).

In Section~\ref{sec_infres} we have defined an action of the double coset operator $[\SS]$ on
cocycles $\vp: F\rar V$. Corollary~\ref{prop_res_hecke} shows that under the restriction
map $\vp\mapsto [\vp(r_i)]_i$, the action of $[\SS]$ on $\vp$ corresponds to the map
on $\oplus_{i=1}^\ell \im (\pp{r_i}{g_i})$ given by
\[
[v_i]_i \mapsto [\sum_{j=1}^\ell v_j|T_\SS^{i,j}]_i.
\]
We conclude that
\[\tr([\SS], \hom(R,V)^F)=\sum_{i=1}^\ell \tr(T_\SS^{i,i}, \im \pp{r_i}{g_i}).\]

Finally we would like to express the left side as the trace of an operator on $V$ rather than on the subspaces
$\im \pp{r_i}{g_i}$. Note that we can view these subspaces also as $\im \pi_i$,
with $\pi_i$ defined in~\eqref{e_pi}. Consider the exact sequence
\[
0\lrar \im (1-\pi_i)\lrar V \xrightarrow{\;v\mapsto v|\pi_i\;} \im(\pi_i)\lrar 0.
\]
We have that $\pps{g_i}{g_i}\cdot \pi_i=\pi_i\cdot T_\SS^{i,i}$ from
the definition of $T_\SS^{i,i}$ in Corollary~\ref{cor_nablaSS}, and
it follows that the operator $\pi_i\pps{g_i}{g_i}$ acting on $V$ corresponds to $T_\SS^{i,i}$
on $\im(\pi_i)$ under the second map above (using that $\pi_i$ is idempotent).
Since the same operator vanishes on $\im (1-\pi_i)$, we have
$$\tr(\pi_i \pps{g_i}{g_i}, V)= \tr(T_\SS^{i,i}, \im \pi_i),$$
and the proof is finished by Proposition~\ref{prop_free} and the inflation-restriction exact sequence.
\end{proof}

\subsection{The case $\G$ has no cusps.} We now assume that $\G$ has no parabolic elements. By a result of Shimura~\cite[Prop 8.2]{Sh} proved by geometric methods we have
\be\label{e_H2}
H^2(\G,V)\simeq V/ \{v|1-g\ \mid\ v\in V, g\in \G\}.
\ee
For simplicity we assume $\G$ has no elliptic elements either, so $h=\ell=0$ in the presentation~\eqref{e_fuchsian}.
Therefore $\G=F/R$ has a presentation with $F$ free on $n=2g$ generators, and $R$ the normal subgroup generated
by the only relator
\[
r=[g_1,g_2]\cdot\ldots\cdot[g_{n-1},g_n].
\]
By Corollary~\ref{cor_nablaSS}, we have $M_\SS\cdot \nabla r=\nabla r\cdot  T_\SS^\vee$, for
an element $T_\SS^\vee\in \Z[\SS]$. Moreover the element $T_\SS^\vee$ is unique,
by~\cite[Sec. 11]{L50}. 
\begin{prop}\label{prop_compact}
Let $\G$ be a finite covolume Fuchsian group with no parabolic and no elliptic elements.
Let $V$ be a $\wG$-module such that~$H^2(\G,V)=0$. We then have:
\[
\sum_{i=0}^2 (-1)^i \tr([\SS], H^i(\G,V))=\tr(\wT_\SS, V),
\]
where
\be\label{e_wtss1}
\wT_\SS=T_\SS+T_\SS^\vee-\sum_{i=1}^n \ppsf{g_i}{g_i}\in \Z[\SS].
\ee
\end{prop}

\begin{proof}
We identify
\be\label{e_homcom}
\hom(R,V)^F=V
\ee
by mapping $f\mapsto f(r)$. Under this identification,
the image of the restriction map in~\eqref{e_infres} is $\sum_{i=1}^n \im(\pp{r}{g_i})$
by Lemma~\ref{lem_cocycle}.

If $H^2(\G,V)=0$, the restriction map in~\eqref{e_infres} is onto. By Proposition~\ref{prop_res_hecke},
the action of $[\SS]$ on
$H^1(F,V)$ corresponds to the action of $T_\SS^\vee$ on $V$ via the identification~\eqref{e_homcom},
and the proof is finished by the exact sequence~\eqref{e_infres} and Proposition~\ref{prop_free}.
\end{proof}

\begin{remark}\label{rem_compact}
By~\eqref{e_H2},
we have $V^\G=V$ if and only if $H^2(\G,V)\simeq V$. In this case we have that restriction
map in~\eqref{e_infres} is trivial, and $\tr([\SS], H^2(\G,V))=\tr([\SS], H^0(\G,V))$ by
Poincar\'e duality \cite[Lemma 1.4.3]{AS}. From Proposition~\ref{prop_free} we obtain
\[
\sum_{i=0}^2 (-1)^i \tr([\SS], H^i(\G,V))=\tr(\wT_\SS', V),\qquad \text{ for }
\wT_\SS'=2T_\SS-\sum_{i=1}^n \ppsf{g_i}{g_i}.
\]
However the operator $\wT_\SS'$, unlike $\wT_\SS$, does not behave nicely with respect to conjugacy classes:
the associated sums $\e'(X)$ defined as in~\eqref{e_eps} will depend on the choice of
representatives~$\ov{\tau}$ made in defining $T_\SS$. In fact we expect that 
$\tr(T_\SS,V)=\tr(T_\SS^\vee,V)$ for $V=V^\G$, that is
\[
\deg T_\SS'=\deg T_\SS=|\G\backslash \SS|,
\]
with $\deg A$ denoting the sum of the coefficients of $A\in\Z[\SS]$.
This would show that Proposition~\ref{prop_compact} holds if $H^2(\G,V)\simeq V$ as well, and
in fact we expect it to hold for arbitrary $V$, but we have been unable to prove it.
 \end{remark}

\subsection{A trace formula}

For an element $A=\sum c_M M\in \Q[\SS]$ and a subset $S\subset \SS$,
denote by $\deg_S A=\sum_{M\in S} c_M$, the sum of the coefficients in $A$
of elements from $S$.

\begin{theorem}\label{thm_main} Let $\G$ be a Fuchsian group of finite covolume, and
assume either $\G$ has cusps, or it has neither cusps nor elliptic elements. Let $V$
be a $\wG$-module, and if $\G$ has no cusps assume also $ H^2(\G,V)=0$.

(a) With the elements $\wT_\SS$ defined in Propositions~\ref{prop_cusps} and~\ref{prop_compact},
we have
\be\label{e_TF}
\sum_{i=0}^2 (-1)^i \tr([\SS], H^i(\G,V)) =\sum_{X\subset \SS} \e(X) \tr(M_X,V)
\ee
where the sum is over  $\G$-conjugacy classes $X\subset \SS$ and
$\e(X)=\deg_X \wT_\SS$.

(b) The coefficients $\e(X)$ are independent of the choice of representatives used to define $T_\SS$.

(c) If $X=\{M\}$ with $\G_M=\G$,  then
$$ \e(X)=- \frac{|\G\backslash\H|}{2\pi},  $$
namely it is equal to the homological Euler-Poincar\'e characteristic of $\G$.
\end{theorem}
\begin{proof}
(a) This follows immediately from Propositions~\ref{prop_cusps} and~\ref{prop_compact},
as $\tr(M,V)$ is constant for $M$ in a conjugacy class $X$.

(b) We show that
$\deg_X \wT_\SS$ is independent of the choice of coset representatives
used to define $T_\SS$.
Any two choices of $T_\SS$ differ by an element $L\in \Z[\SS]$ with
$\deg_\tau L=0$ for all $\tau\in \G\backslash\SS$.
Therefore we can write
\be\label{e_L}
L=\sum_{j=1}^r
(1-g_j)Y_j
\ee
for some $Y_j\in \Z[\SS]$. Now the elements $\pps{g_i}{g_j}$ for the two choices
differ by elements $L_{ij}$ such that
$
L(1-g_i)=\sum_{j=1}^n (1-g_j) L_{ij}.
$

Assume first that $\G$ has cusps, so $\wT_\SS$ is given by~\eqref{e_wtss}.
It follows that the two corresponding $\wT_\SS$
differ by an element
$$\wL=L-\sum_{i=1}^n L_{ii}+\sum_{i=1}^l \pi_i L_{ii} .$$
Multiplying~\eqref{e_L} on the right by $(1-g_i)$ and
using Proposition~\ref{prop_lyndonSS} we obtain
\[
L_{ii}=Y_i(1-g_i)+\pi_i Z_i \text{ for } i\le l, \quad 
L_{ii}=Y_i(1-g_i) \text{ for } i>l\;,
\]
for some $Z_i\in\Q[\SS]$. We conclude that
\[
\wL=\sum_{i=1}^n (Y_i g_i-g_i Y_i)+\sum_{i=1}^l \pi_i Y_i (1-g_i)\;,
\]
and it is clear that each term in both sums has degree 0 over each 
conjugacy class.

Assume now $\G$ has no cusps nor elliptic elements, and $\wT_\SS$ is
given by~\eqref{e_wtss1}. Proposition~\ref{prop_lyndonSS} now gives
\[L_{ij}=Y_j(1-g_i)+\ppf{r}{g_j} Z_i,
\]
for some $Z_i\in\Z[\SS]$. By Corollary~\ref{cor_nablaSS} it follows that the elements
$\wT_\SS^\vee$ corresponding to the two choices of representatives differ by $L^\vee$ such that
\[\begin{aligned}
\sum_{j=1}^n L_{ji}\ppf{r}{g_j}&=\ppf{r}{g_i}L^\vee \\
&=\sum_j Y_i(1-g_j)\pp{r}{g_j}+\sum_j \pp{r}{g_i}Z_j\pp{r}{g_j},
\end{aligned}
\]
where in the second line we used the formula above for $L_{ij}$. The first sum vanishes since $r=1$ in $\G$,
and we conclude that
$L^\vee-\sum_j Z_j\pp{r}{g_j} $ is annihilated by left multiplication by $\pp{r}{g_i}$.

The explicit formula for $r$ gives $ \pp{r}{g_i}=(1-g_{i'})h_j$ for $i'=i\pm 1$ and $h_i\in \G$.
If $g$ is of infinite order, we have $\ker (1-g)=0$: if $(1-g)\sum c_M M=0$ in $\Q[\SS]$,
it follows that $c_M=c_{Mg}=c_{Mg^n}$ for all $n$, and if there is $c_M\ne 0$, one would have
infinitely many coefficients $c_{Mg^n}\ne 0$, a contradiction.

It follows that $L^\vee=\sum_j Z_j\pp{r}{g_j} $, and so the two corresponding $\wT_\SS$
differ by
\[\wL=L+L^\vee-\sum_{i=1}^n L_{ii}=\sum_{i}(Y_ig_i-g_iY_i)
+\sum_i Z_i\ppf{r}{g_i}-\ppf{r}{g_i}Z_i,
\]
and clearly $\deg_X \wL=0$ for every conjugacy class $X$, finishing the proof.

\comment{
When $\G$ is free, it can be easily shown that $\deg_X \wT_\SS$
does not depend on the generators $\{ g_1,\ldots, g_n\}$ of
$\G$. Indeed any other basis can be obtained from the given one by 
two types of elementary transformations that leave all 
basis elements unchanged except for one $g_i$, which is mapped as 
follows: (1) $g_i \mapsto 
g_i^{-1}$; (2)  $g_i\mapsto g_ig_j$ for some $j\ne i$. Denoting by 
$\wT_\SS'$ the 
element defined in~\eqref{7} with respect to the new basis (for the 
same  $T_\SS$), it is easy to check 
that  
\[\wT_\SS'-\wT_\SS=
\begin{cases} g_i T_{ii}g_i^{-1} -T_{ii} & \text{in case (1)}\\ 
              g_j^{-1} T_{ii} g_j-T_{ii} & \text{in case (2)}
\end{cases},
\]
so $\deg \wT_{\SS}'^{(\X)}=\deg \wT_{\SS}^{(\X)}$ for all conjugacy 
classes $\X$. 

A similar reasoning applies when some of the $g_i$ 
have finite order, when any set of generators can be obtained from  
the given one by the two elementary transformations above when $g_i$
have infinite order, and when the order of $g_i$ is $n$ by the 
following transformations that leave all generators fixed except 
for $g_i$: (3) $g_i\mapsto g_i^a$ with $a$ coprime to $n$; (4) 
$g_i\mapsto g_j^{-1} g_i g_{j}$ for $j\ne i$.  
}
(c) Since formula \eqref{e_TF} is additive under taking union of double
cosets, it is enough to prove the formula for $\e(X)$ when
$\SS=\G M$ and $X=\{M\}$. In this
case we have
$T_\SS=M$, and it follows that $\pps{g_i}{g_i}=M$ for all $i$.
Also $T_\SS^\vee=M$ if $\G$ has no cusps, and the
coefficient of~$M$ in~\eqref{e_wtss} or \eqref{e_wtss1} gives the formula for $\e(X)$,
by the Gauss-Bonnet formula~\eqref{e_chi}.
\end{proof}
For the trivial double coset $\SS=\G$, the operators $\wT_\G$ can be explicitly computed:
\[\wT_\G=1-n-\sum_{i=1}^n \pi_i,\text{ or }\ \wT_\G=2-n
\]
depending on whether $\G$ has cusps or not. If $\G$ has cusps, the elements 
$g_i^a$ for $1\le i\le l$, $0\le a<m_i$
form a set of representatives  for conjugacy classes of elliptic elements, which we denote $E(\G)$, and we obtain.
\begin{cor}\label{cor_dim}
If $\G$ has cusps, let $V$ be any $\wG$-module, while if $\G$ has no cusps nor elliptic elements
assume either $H^2(\G,V)=0$ or $H^2(\G,V)=V$. Then
\[\sum_i (-1)^i \dim H^i(\G,V)=-\frac{|\G\backslash \H|}{2\pi}\dim V + 
\sum_{g\in E(\G)} \frac{1}{|\G_g|}\tr(g, V) ,
\]
\end{cor}
\begin{proof}
Only the case  $H^2(\G,V)=V$ when $\G$ has no parabolic elements requires justification.  In this case the formula follows from Remark~\ref{rem_compact}.
\end{proof}

\comment{
\subsection{Computing $\e(X)$}

In the previous sections we have defined invariants $\e(X)$ for $\G$-conjugacy classes $X$ contained
in a double coset $\SS$ of $\G$, using the operators $\wT_\SS$.
For $M\in \SS$, we let $\DD(M)=\tr^2 M-4 \det M$ be the discriminant of the associated quadratic form,
and we denote by $\G_M$ the centralizer of $M$ in $\G$.

We have the following conjecture.
\begin{conj*}We have
\[\e(X)=\begin{cases}-\frac{\sgn \DD(M)}{|\G_M|}\\
         -\chi(\G)
        \end{cases}
\]
where $M\in X$ is any representative and $\chi(\G)$ has been defined in~\eqref{e_chi}.
\end{conj*}
If $\G_M$ is infinite, we set $\e(X)=0$.

\begin{prop} \label{prop_shapiro}
Assume the group $\G$ has a trace formula of the type:
\[
\sum_{i=0}^n (-1)^i \tr([\SS], H^i(\G,V)) =\sum_{M\in  \SS//\G}\e(M) \tr(M,V)
\]
for every double coset $\SS$ in a commensurator $\wG$ and for every module $V$ over a
field in which the orders of elements of $\G$ are
invertible. If $\G'$ is a finite index subgroup of $\G$, and $\SS'$ a double coset of
$\G'$ such that
$$|\G'\backslash \SS'|=|\G\backslash\G \SS'|,$$
then we have
\[
\sum_{i=0}^n (-1)^i \tr([\SS'], H^i(\G',V)) =\sum_{M\in  \SS'//\G'}\e'(M) \tr(M,V)
\]
for every $\wG$-module $V$, where
\[\e'(M)=\e(M)\cdot [\stab_{\G} M: \stab_{\G'} M ].
\]
\end{prop}
}

\subsection{The modular group}\label{sec_modular} For $\G=\PSL_2(\Z)$, let $g_1, g_2$ be
generators of orders 2, 3 respectively. Let $\SS=\SS_n$ be the double coset of
integral matrices of determinant $n\ge 1$. In~\cite{PZ}, together with Zagier we
defined elements $\wT_n$ giving the action of the Hecke operator $[\SS_n]$
on period polynomials of modular forms, and satisfying an extra property that
allowed us to prove a trace formula as in Proposition~\ref{prop_cusps},
with $\wT_\SS$ replaced by $-\wT_n$.

The element $\wT_\SS$ defined in~\eqref{e_wtss} does not preserve the space of period polynomials
for~$\G$, but we can adjust it so that it does as follows.
Let $X_i\in \Q[\G]$ such that  $1-\pi_i=(1-g_i)X_i$ (that is $X_1=\frac 12$,
$X_2=\frac{2+g_2}{3}$), and define:
\[
\wT_\SS':=T_\SS-\sum_{i=1}^2 (1-g_i)\ppsf{g_i}{g_i} X_i.
\]
It is easily verified that $\deg_X \wT_\SS=\deg_X \wT_\SS'$ for all conjugacy classes $X$.
For $i\ne j\in \{1,2\}$,  we have by the definition~\eqref{e_hecke}
\[\pi_i T_\SS (1-\pi_j)=\pi_i T_\SS (1-g_j)X_j= \pi_i (1-g_j)\ppsf{g_j}{g_j} X_i
\]
so $\pi_i \wT_\SS'= \pi_iT_\SS\pi_j\in \im \pi_j$.

Since $\deg_\tau {\wT_\SS'}=1$ for all cosets $\tau\in
\G\backslash\SS$, it follows from~\cite[Corollary 2]{PZ} that the
adjoint of $-\wT_\SS'$ satisfies properties (A), (B), (C) introduced
for the modular group, so it is one of our previously constructed operators $\wT_n$.
Therefore the constants $\e(X)$ in Theorem \ref{thm_main}
agree with $\w(X)$ defined in the introduction, as shown
in~\cite{PZ} by giving an example of such an element $\wT_n$.


\begin{thebibliography}{9999}
\bibitem{AS} A. Ash, G. Stevens, \emph{Cohomology of arithmetic groups and congruences
between systems of Hecke eigenvalues}. J. Reine Angew. Math. 365 (1986), 192--220

\bibitem{Ba} H. Bass, \emph{Euler characteristics and characters of
discrete groups.} Inv. Math. 35 (1976), 155--196

\bibitem{Br} K.S. Brown, \emph{Complete Euler Characteristics and fixed-point
theory.} J. Pure and Applied Algebra 24 (1982), 103--121

\bibitem{CZ} Y.J. Choie, D. Zagier, \emph{Rational period functions for
PSL(2,Z).} Contemporary Math. 143 (1993), 89--108
\bibitem{F53} R.H. Fox, \emph{Free differential calculus, I. Derivation
 in the free group ring}. Ann. of Math. 57 (1953), 547--560

\bibitem{GM1}M. Goresky, R. MacPherson, \emph{Lefschetz numbers of Hecke correspondences.} in: 
The Zeta Functions of Picard Modular Surfaces, R. Langlands and D. Ramakrishnan, eds., 
Publ. C.R.M., Univ. de Montr\'eal (1992), 465--478.
 
 
\bibitem{GM2}  M. Goresky, R. MacPherson, \emph{The topological trace formula}. 
J. reine angew. Math. 560 (2003), 77--150

 
\bibitem{L50} R.C. Lyndon, \emph{Cohomology Theory of Groups with a Single Defining Relation}.
Annals of Math. 52 (1950), 650--665

\bibitem{O77} J. Oesterl\'e, \emph{Sur la trace des op\'erateurs de Hecke}. Th\`ese de troisi\`eme cycle,
Universit\'e Paris-Sud, Orsay (1977)

\bibitem{Sh} G. Shimura, \emph{Introduction to the arithmetic theory of automorphic
functions}. Princeton U. Press (1971)
\bibitem{PZ} A.A. Popa, D. Zagier, \emph{An elementary proof of the 
Eichler-Selberg trace formula}. J. Reine Angew. Math. 762 (2020), 105--122
\bibitem{P} A.A. Popa, \emph{On the trace formula
for Hecke operators on congruence subgroups}. Proc. Amer. Math. Soc. 146/7
(2018), 2749--2764
\bibitem{P2} A.A. Popa, \emph{On the trace formula
for Hecke operators on congruence subgroups, II}. Research in the Math. Sciences (2018) 5:3
\bibitem{Se} J.-P. Serre, \emph{Cohomologie des groupes discrets.} Ann. Math. Studies 70, Princeton Univ.
Press, 77--169 

\bibitem{S69} R.G. Swan, \emph{Groups of Cohomological Dimension One}. J. Algebra 21 (1969), 585--601

\bibitem{Z} D. Zagier, \emph{Periods of modular forms, traces of
Hecke operators, and multiple zeta values.} Research into automorphic forms and
$L$ functions (Japanese) (Kyoto, 1992).  S\=urikaisekikenky\=usho K\=oky\=uroku 
No. 843  (1993), 162--170

\end{thebibliography}
\end{document}